\documentclass{amsart}

\usepackage{ae,aecompl}
\usepackage{chngcntr}
\usepackage{graphicx}
\usepackage[all,cmtip]{xy}
\usepackage{amsmath, amscd}
\usepackage{amsthm}
\usepackage{amssymb}
\usepackage{amsfonts}
\usepackage{qsymbols}
\usepackage{latexsym}
\usepackage{mathrsfs}
\usepackage{cite}
\usepackage{color}
\usepackage{url}
\usepackage{enumerate}
\usepackage[utf8]{inputenc}
\usepackage[T1]{fontenc}
\usepackage{verbatim}
\usepackage[draft=false, colorlinks=true]{hyperref}

\allowdisplaybreaks

\newcommand {\abs}[1]{\lvert#1\rvert}

\newcommand {\B}{{\mathcal{B}}}

\newcommand {\C}{{\mathbb C}}

\newcommand {\D}{D}

\newcommand {\ud}{\, \mathrm{d}}
\newcommand {\ue}{e}

\newcommand {\Ell}{L}
\newcommand {\Ellp}{L^{p}}

\newcommand {\F}{{\mathcal{F}}}

\newcommand {\ui}{i}
\newcommand {\I}{{I}}

\newcommand {\La}{{\mathcal{L}}}
\newcommand {\calL}{{\mathcal{L}}}

\newcommand {\Ma}{{\mathcal{M}}}
\newcommand {\N}{{{\mathbb N}}}

\newcommand {\norm}[1]{\left\|#1\right\|}
\newcommand{\one}{{{\bf 1}}}
\newcommand {\ph}{{\varphi}}

\newcommand {\R}{{\mathbb R}}

\newcommand {\Rn}{{\mathbb{R}^{n}}}

\newcommand {\Se}{\mathrm{S}}

\newcommand {\Sw}{\mathcal{S}}
\newcommand {\Schw}{\mathcal{S}}
\newcommand {\T}{{\mathbb T}}
\newcommand {\w}{{\omega}}

\newcommand{\wh}{\widehat}

\newcommand {\Z}{{{\mathbb Z}}}

\newcommand {\vanish}[1]{\relax}

\DeclareMathOperator{\Real}{Re}

\newtheorem{theorem}{Theorem}[section]
\newtheorem{lemma}[theorem]{Lemma}
\newtheorem{proposition}[theorem]{Proposition}
\newtheorem{corollary}[theorem]{Corollary}

\theoremstyle{definition}

\newtheorem{remark}[theorem]{Remark}
\newtheorem{example}[theorem]{Example}

\numberwithin{equation}{section}
\protected\def\ignorethis#1\endignorethis{}
\let\endignorethis\relax

\title{Sharp growth rates for semigroups using resolvent bounds}

\author{Jan Rozendaal}
\address{Mathematical Sciences Institute\\ Australian National University\\ Acton ACT 2601\\ Australia\\ and Institute of Mathematics, Polish Academy of Sciences\\
ul.~\'{S}niadeckich 8\\
00-656 Warsaw\\
Poland}
\email{janrozendaalmath@gmail.com}

\author{Mark Veraar}
\address{Delft Institute of Applied
Mathematics\\
Delft University of Technology\\
P.O.~Box 5031\\
2628 CD Delft\\
The Netherlands}
\email{M.C.Veraar@tudelft.nl}



\thanks{The first author is supported by grant DP160100941 of the Austalian Research Council. The second author is supported by the VIDI subsidy 639.032.427 of the Netherlands Organisation for Scientific Research (NWO)}

\begin{document}

\begin{abstract}
We study growth rates for strongly continuous semigroups. We prove that a growth rate for the resolvent on imaginary lines implies a corresponding growth rate for the semigroup if either the underlying space is a Hilbert space, or the semigroup is asymptotically analytic, or if the semigroup is positive and the underlying space is an $L^{p}$-space or a space of continuous functions. We also prove variations of the main results on fractional domains; these are valid on more general Banach spaces.
In the second part of the article we apply our main theorem to prove optimality in a classical example by Renardy of a perturbed wave equation which exhibits unusual spectral behavior.
\end{abstract}

\maketitle

\section{Introduction}

Let $-A$ be the generator of a $C_{0}$-semigroup $(T(t))_{t\geq0}$ on a Banach space $X$. It can be quite difficult to verify the assumptions of the Hille--Yosida theorem to determine whether $(T(t))_{t\geq0}$ is uniformly bounded, given that bounds for all powers of the resolvent of $A$ are required. Hence it is of interest to determine spectral conditions that are easier to check and which imply specific growth behavior of $(T(t))_{t\geq0}$, such as for example polynomial growth. One such condition is the Kreiss resolvent assumption from \cite{Kreiss59}: $\sigma(A)\subseteq\overline{\C_{+}}$ and
\begin{equation}\label{eq:Kreiss condition}
\|(\lambda+A)^{-1}\|\leq \frac{K}{\Real(\lambda)}\qquad (\lambda\in\C_{+})
\end{equation}
for some $K\geq0$. It is known from \cite{vDoKrSp93} that \eqref{eq:Kreiss condition} implies $\|T(t)\|\leq enK$ if $X$ is $n$-dimensional. Moreover, as was shown in \cite{Eisner-Zwart06}, if $X$ is a Hilbert space and \eqref{eq:Kreiss condition} holds then $\|T(t)\|$ grows at most linearly in $t$, while there exist semigroups on general Banach spaces which satisfy \eqref{eq:Kreiss condition} but grow exponentially. For more on this topic see \cite{vDoKrSp93, Eisner-Zwart06, Strikwerda-Wade97} and references therein.

There are many interesting strongly continuous semigroups with a polynomial growth rate. One important class is given by certain Schr\"odinger semigroups on $L^{p}$-spaces, $p\in[1,\infty]$, that have generator $\Delta+V$ for $V$ an (unbounded) potential (see \cite{DaSi, Grig} and references therein). Other examples arise from (perturbed) wave equations \cite{GoldW,Paunonen14}, delay equations \cite{Sklyar-Polak17}, and operator matrices and multiplication operators \cite[Section 4.7]{Rozendaal-Veraar18b}. In \cite{ArBaHiNe11,Daviesbook,Eisner10,Engel-Nagel00,vanNeerven96b} and \cite{Dav05} one may find additional examples of semigroups with interesting growth behavior.

The following is the main result of this article. It enables one to derive polynomial growth bounds for a semigroup from resolvent estimates similar to \eqref{eq:Kreiss condition}. We note that each eventually differentiable $C_{0}$-semigroup, and in particular each analytic semigroup, is asymptotically analytic. Also, condition \eqref{it:mainC} is satisfied if e.g.~$X=C_{ub}(\Omega)$ for $\Omega$ a metric space, or $X=C_{0}(\Omega)$ for $\Omega$ a locally compact space.

\begin{theorem}\label{thm:main}
Let $-A$ be the generator of a $C_{0}$-semigroup $(T(t))_{t\geq0}$ on a Banach space $X$ such that $\C_-\subseteq \rho(A)$. Assume that one of the following conditions holds:
\begin{enumerate}
\item $X$ is a Hilbert space;
\item $(T(t))_{t\geq0}$ is an asymptotically analytic semigroup;
\item $X = L^p(\Omega)$ for $p\in [1, \infty)$ and $\Omega$ a measure space, and $T(t)$ is a positive operator for all $t\geq0$.
\item\label{it:mainC} $X$ is a closed subspace of $C_{b}(\Omega)$, for $\Omega$ a topological space, such that either $\mathbf{1}_{\Omega}\in X$ or $X$ is a sublattice, and $T(t)$ is a positive operator for all $t\geq0$.
\end{enumerate}
If there exist $\alpha\in [0, \infty)$ and $K\geq 1$ such that
\begin{equation}\label{eq:R}
\|(\lambda + A)^{-1}\|_{\La(X)}\leq K (\Real(\lambda)^{-\alpha} +1)\qquad (\lambda\in\C_+),
\end{equation}
then there exists a $C\geq0$ such that
\begin{equation}\label{eq:T}
\|T(t)\|_{\La(X)} \leq  CK (t^{\alpha} + 1)\qquad (t\geq0).
\end{equation}
\end{theorem}

In fact, in the main text we allow an arbitrary growth rate $g$ in \eqref{eq:R} and \eqref{eq:T}.
It follows from Example \ref{ex:Hilbert} below that, for $\alpha\in\N$, Theorem \ref{thm:main} is optimal up to arbitrarily small polynomial loss in \eqref{eq:T}.

For $\alpha = 0$ and $X$ a Hilbert space, Theorem \ref{thm:main} reduces to the Gearhart-Pr\"uss theorem (see \cite[Theorem 5.2.1]{ArBaHiNe11}), while for $\alpha=0$ and $(T(t))_{t\geq0}$ a positive semigroup on an $L^{p}$-space one recovers a result by Weis (see \cite[Theorem 5.3.1]{ArBaHiNe11}).

For $\alpha\in (0,1)$ the inequality $\|R(\lambda, A)\|\geq \text{dist}(\lambda, \sigma(A))$ for $\lambda\in \rho(A)$ shows that $\overline{\C_-}\subseteq \rho(A)$, and then one can use a Neumann series argument to reduce to the case where $\alpha=0$.

For $\alpha\geq 1$ it was previously known from \cite{Eisner-Zwart07} that \eqref{eq:R} implies
\begin{equation}\label{eq:eisnerzwart}
\|T(t)\|_{\La(X)}\leq C K(t^{2\alpha-1}+1)\qquad (t\geq0)
\end{equation}
whenever $(T(t))_{t\geq0}$ has a so-called \emph{$p$-integrable resolvent} for some $p\in(1,\infty)$. This property is satisfied by e.g.~all $C_{0}$-semigroups on Hilbert spaces and analytic semigroups on general Banach spaces. If $\alpha=1$ then \eqref{eq:T} and \eqref{eq:eisnerzwart} yield the same conclusion. In all other cases \eqref{eq:T} improves \eqref{eq:eisnerzwart}. Theorem \ref{thm:main} also seems to be the first result of its kind for asymptotically analytic semigroups and for positive semigroups on $L^{p}$-spaces and spaces of continuous functions. Generation theorems for (semi)groups with polynomial growth were discussed in \cite{Eisner10, Kis, Mal01}. In contrast to these articles we assume a priori that the relevant semigroup exists. Other results on semigroups of polynomial growth can be found in \cite{Boukdir15, Eisner-Zwart06, vanNeerven09}. Versions of Theorem \ref{thm:main} for C\'esaro type averages have been considered in \cite{LiSaSh08}, where also numerous counterexamples are presented.

It was known from \cite{Eisner-Zwart07} that on general Banach spaces \eqref{eq:T} implies
\[
\|(\lambda + A)^{-1}\|_{\La(X)}\leq C'(\Real(\lambda)^{-\alpha-1} + 1 )\qquad (\lambda\in \C_+)
\]
for some $C'\geq0$, thus providing a partial converse to Theorem \ref{thm:main}. In Theorem \ref{thm:abstractconverse} and Corollary \ref{cor:abstractchar} we extend this result and obtain a full characterization of polynomial stability of a semigroup in terms of properties of the resolvent of its generator.

We also derive versions of Theorem \ref{thm:main} on fractional domains, where we make other geometric assumptions on $X$. In particular, it is shown in Proposition \ref{prop:polynomialgrowthgeneral} that on a general Banach space $X$ \eqref{eq:Kreiss condition} implies at most linear growth for semigroup orbits with sufficiently smooth initial values. We also point out that, by choosing $\alpha=0$ and using a scaling argument, Theorem \ref{thm:main} and other results in Section \ref{sec:growth} imply various theorems about exponential stability from \cite{vanNeerven96a,Weis97,vanNeerven09,Weis-Wrobel96}.

We note here that the main result of \cite{Eisner-Zwart06} was applied to Schr\"odinger semigroups in \cite[Theorem 5.4]{Faupin-Frohlich17} to deduce cubic growth of the semigroup, whereas Theorem \ref{thm:main} immediately yields quadratic growth.

To prove Theorem \ref{thm:main} we use the connection between stability theory and Fourier multipliers which goes back to e.g.~\cite{Kaashoek-VerduynLunel94,Hieber99,Weis97,Latushkin-Shvydkoy01} and which was renewed in \cite{Rozendaal-Veraar18b}, following the development of a theory of operator-valued $(L^{p},L^{q})$ Fourier multipliers in \cite{Rozendaal-Veraar17,Rozendaal-Veraar18}. In particular, Theorem \ref{thm:abstract polynomial growth} gives a Fourier multiplier criterion for a bound as in \eqref{eq:T} to hold, and Corollary \ref{cor:abstractchar} gives a characterization of polynomial growth and uniform boundedness of a semigroup in terms of multiplier properties of the resolvent. Theorem \ref{thm:main} is then deduced using Plancherel's theorem, known connections between Fourier multipliers and analytic semigroups from \cite{Batty-Srivastava03}, and a Fourier multiplier theorem for positive kernels from Proposition \ref{prop:positivekernel}.

In Section \ref{sec:exRen} we apply Theorem \ref{thm:main} to obtain optimality of the growth rate in a perturbed wave equation which was studied by Renardy in \cite{Renardy94} and which exhibits unusual spectral behavior.

\section{Notation and preliminaries}

We denote by $\C_{+}:=\{\lambda\in\C\mid \Real(\lambda)>0\}$ and $\C_{-}:=-\C_{+}$ the open complex right and left half-planes.

Nonzero Banach spaces over the complex numbers are denoted by $X$ and $Y$. The space of bounded linear operators from $X$ to $Y$ is $\La(X,Y)$, and $\La(X):=\La(X,X)$. The identity operator on $X$ is denoted by $\I_{X}$, and we usually write $\lambda$ for $\lambda\I_{X}$ when $\lambda\in\C$.
The domain of a closed operator $A$ on $X$ is $\D(A)$, a Banach space with the norm
\begin{align*}
\norm{x}_{\D(A)}:=\norm{x}_{X}+\norm{Ax}_{X}\qquad (x\in \D(A)).
\end{align*}
The spectrum of $A$ is $\sigma(A)$ and the resolvent set is $\rho(A)=\C\setminus \sigma(A)$. We write $R(\lambda,A)=(\lambda -A)^{-1}$ for the resolvent operator of $A$ at $\lambda\in\rho(A)$.

For $p\in[1,\infty]$ and $\Omega$ a measure space, $\Ellp(\Omega;X)$ is the Bochner space of equivalence classes of strongly measurable, $p$-integrable, $X$-valued functions on $\Omega$.
The H\"{o}lder conjugate of $p\in[1,\infty]$ is $p'\in[1,\infty]$ and is defined by $1=\frac{1}{p}+\frac{1}{p'}$.

The indicator function of a set $\Omega$ is denoted by $\mathbf{1}_{\Omega}$. We often identify functions on $[0,\infty)$ with their extension to $\R$ which is identically zero on $(-\infty,0)$.

The class of $X$-valued Schwartz functions on $\Rn$, $n\in\N$, is denoted by $\Sw(\Rn;X)$, and $\Sw(\Rn):=\Sw(\Rn;\C)$. The space of continuous linear $f:\Sw(\Rn)\to X$, the $X$-valued tempered distributions, is $\Sw'(\Rn;X)$. The Fourier transform of $f\in\Sw'(\Rn;X)$ is denoted by $\F f$ or $\widehat{f}$. If $f\in\Ell^{1}(\Rn;X)$ then
\begin{align*}
\F f(\xi)=\wh{f}(\xi)=\int_{\Rn}\ue^{-\ui \xi\cdot t}f(t)\ud t\qquad (\xi\in\Rn).
\end{align*}

Let $X$ and $Y$ be Banach spaces. A function $m:\Rn\to\La(X,Y)$ is \emph{$X$-strongly measurable} if $\xi\mapsto m(\xi)x$ is a strongly measurable $Y$-valued map for all $x\in X$. We say that $m$ is \emph{of moderate growth} if there exist $\alpha\in(0,\infty)$ and $g\in \Ell^1(\R)$ such that
\[
(1+\abs{\xi})^{-\alpha} \|m(\xi)\|_{\La(X,Y)} \leq g(\xi) \qquad(\xi\in\Rn).
\]
Let $m:\Rn\to\La(X,Y)$ be an $X$-strongly measurable map of moderate growth. Then $T_{m}:\Sw(\Rn;X)\to\Sw'(\Rn;Y)$,
\begin{equation}\label{eq:Fouriermultiplier}
T_{m}(f):=\F^{-1}(m\cdot\widehat{f}\,)\qquad(f\in\Sw(\Rn;X)),
\end{equation}
is the \emph{Fourier multiplier operator} associated with $m$. For $p\in[1,\infty)$ and $q\in[1,\infty]$ we let $\Ma_{p,q}(\Rn;\La(X,Y))$ be the set of all $X$-strongly measurable $m:\Rn\to\La(X,Y)$ of moderate growth such that $T_{m}\in\La(L^p(\Rn;X),L^{q}(\Rn;Y))$, with
\[
\|m\|_{\Ma_{p,q}(\Rn;\La(X,Y))}:=\|T_{m}\|_{\La(L^{p}(\Rn;X),L^{q}(\Rn;Y))}.
\]
Moreover, suppose that there exists an $X$-strongly measurable $K:\Rn\to\La(X,Y)$ such that $K(\cdot)x\in L^{1}(\Rn;Y)$ and $m(\xi)x=\F(K(\cdot)x)(\xi)$ for all $x\in X$ and $\xi\in\Rn$. Then for $f\in L^{\infty}(\Rn)\otimes X$ an $X$-valued simple function one may define
\[
T_{m}(f)(t):=\int_{\Rn}K(t-s)f(s)\ud s \qquad(t\in\Rn).
\]
We write $m\in \Ma_{\infty,\infty}(\Rn;\La(Y,X))$ if there exists a constant $C\geq 0$ such that
\begin{equation}\label{eq:inftymultiplier}
\|T_{m}(f)\|_{L^{\infty}(\Rn;Y)}\leq C\|f\|_{L^{\infty}(\Rn;X)}
\end{equation}
for all such $f$, and then we let $\|m\|_{\Ma_{\infty,\infty}(\Rn;\La(X,Y))}$ be the minimal constant $C$ in \eqref{eq:inftymultiplier}. In this case $T_{m}$ extends to a bounded operator from the closure of the $X$-valued simple functions in $L^{\infty}(\Rn;X)$ to $L^{\infty}(\Rn;Y)$. This closure is not in general equal to $L^{\infty}(\Rn;X)$, but for $n=1$ it contains all regulated functions (e.g.\ piecewise continuous $f$) that vanish at infinity
(see \cite[7.6.1]{Dieu69}), which will suffice for our purposes.

For $\ph\in(0,\pi)$ set
\[
S_{\ph}:=\{z\in \C\setminus\{0\}\mid \abs{\arg(z)}<\ph\}.
\]
A operator $A$ on a Banach space $X$ is \emph{sectorial of angle} $\ph\in(0,\pi)$ if $\sigma(A)\subseteq \overline{S_{\ph}}$ and if $\sup\{\|\lambda R(\lambda,A)\|_{\La(X)}\mid \lambda\in \C\setminus\overline{S_{\theta}}\}<\infty$ for all $\theta\in(\ph,\pi)$. An operator $A$ such that
\[
M(A):=\sup\{\|\lambda(\lambda+A)^{-1}\|_{\La(X)}\mid \lambda\in(0,\infty)\}<\infty
\]
is sectorial of angle $\ph=\pi-\arcsin(1/M(A))$, and for each $\theta>\pi-\arcsin(1/M(A))$ there exists a constant $C_{\theta}\geq0$ independent of $A$ such that
\begin{equation}\label{eq:sectorial}
\sup\{\|\lambda R(\lambda,A)\|_{\La(X)}\mid \lambda\in \C\setminus\overline{S_{\theta}}\}\leq C_{\theta}M(A),
\end{equation}
as follows from the proof of \cite[Proposition 2.1.1.a]{Haase06a}. For $-A$ the generator of a $C_{0}$-semigroup $(T(t))_{t\geq0}\subseteq\La(X)$ on a Banach space $X$, set
\begin{align*}
\w_{0}(T):=\inf\{\w\in\R\mid \exists M\geq 0: \|T(t)\|_{\La(X)}\leq M\ue^{\w t}\text{ for all }t\geq0\}
\end{align*}
and $s(-A):=\sup\{\Real(\lambda)\mid\lambda\in\sigma(-A)\}$. Then $\w+A$ is a sectorial operator for $\w>\w_{0}(T)$. In particular, for $\gamma\in[0,\infty)$ the fractional domain $X_{\gamma} := D((\omega+A)^{\gamma})$ is well defined, and up to norm equivalence it is independent of the choice of $\w$. For background knowledge on $C_{0}$-semigroups and sectorial operators we refer to \cite{ArBaHiNe11,Eisner10,Haase06a,Engel-Nagel00,vanNeerven96b}.

\section{Polynomial growth results}\label{sec:growth}

Throughout this section, for $-A$ the generator of a $C_0$-semigroup $(T(t))_{t\geq0}$ on a Banach space $X$, let $\omega, M_{\omega}\geq1$ be such that
\begin{equation}\label{eq:assTM}
\|T(t)\|_{\calL(X)} \leq M_{\omega} e^{t(\omega-1)} \qquad(t\geq0),
\end{equation}
and set $M := \sup\{\|T(t)\|_{\calL(X)}\mid t\in [0,2]\}$.

\subsection{General Banach spaces}\label{subsec:growthmultipliers}

We first consider semigroups on general Banach spaces. In \cite{Eisner-Zwart07} an example is given of a semigroup generator $-A$ which satisfies \eqref{eq:Kreiss condition} such that the associated semigroup grows exponentially. The following proposition shows in particular that the Kreiss condition does imply at most linear growth of semigroup orbits with sufficiently smooth initial values.

\begin{proposition}\label{prop:polynomialgrowthgeneral}
Let $-A$ be the generator of a $C_{0}$-semigroup $(T(t))_{t\geq0}$ on a Banach space $X$ such that $\C_-\subseteq\rho(A)$. Suppose that there exists a nondecreasing $g:(0,\infty)\to (0,\infty)$ such that
\[
\|(\lambda+ A)^{-1}\|_{\La(X)} \leq g(\Real(\lambda)^{-1})\qquad (\lambda\in\C_{+}).
\]
Then for each $\gamma\in(1,\infty)$ there exists a $C_{\gamma}>0$ such that $\|T(t)\|_{\La(X_{\gamma},X)}\leq C_{\gamma}g(t)+M$ for all $t>0$.
\end{proposition}
\begin{proof}
It suffices to prove the estimate for $t\geq 2$. Let $x\in X_{\gamma}$ and set $y:=(1+A)^{\gamma}x\in X$. For $a\in (0,1)$ the functional calculus for half-plane operators from \cite{BaHaMu13} yields
\[
e^{-at}T(t)x=\frac{1}{2\pi \ui}\int_{i\R}\frac{e^{-zt}}{(1-a+z)^{\gamma}}R(z,A+a)y\ud z.
\]
Hence there exists a constant $C'_{\gamma}>0$ such that, for all $a\in (0,\tfrac12)$, 
\[
\|T(t)x\|_{X}\leq \frac{1}{2\pi}e^{at}g(1/a)\|y\|_{X}\!\int_{\ui\R}\frac{1}{|1-a+z|^{\gamma}}|\text{d}z|\leq C'_{\gamma}e^{at}g(1/a)\|x\|_{X_{\gamma}}.
\]
Now set $a=1/t$ to conclude the proof.
\end{proof}

The following theorem is inspired by \cite[Theorem 3.1]{Latushkin-Shvydkoy01}. It links growth rates of a semigroup to the Fourier multiplier properties of the resolvent of its generator.

\begin{theorem}\label{thm:abstract polynomial growth}
Let $-A$ be the generator of a $C_{0}$-semigroup $(T(t))_{t\geq0}$ on a Banach space $X$ such that $\C_{-}\subseteq\rho(A)$, and let $Y\hookrightarrow X$ be a continuously embedded Banach space satisfying the following conditions:
\begin{enumerate}[(1)]
\item\label{it:condabstract polynomial growth1} There exists a $C_{T}\geq0$ such that $T(t)\in\La(Y)$ for all $t\geq0$, with $\|T(t)\|_{\calL(Y)}\leq C_{T} \|T(t)\|_{\calL(X)}$;
\item\label{it:condabstract polynomial growth2} There exists a continuously and densely embedded Banach space $Y_0\hookrightarrow Y$ such that $[t\mapsto e^{-a t} \|T(t)\|_{\La(Y_0, X)}]\in L^1(0,\infty)$ for all $a\in(0,\infty)$.
\end{enumerate}
Suppose that there exist $p\in[1,\infty]$, $q\in[p,\infty]$ and a nondecreasing $g:(0, \infty)\to (0,\infty)$ such that $(a+\ui\cdot+A)^{-1}\in \Ma_{p,q}(\R;\La(Y,X))$ for all $a\in(0,\infty)$, with
\begin{equation}\label{eq:pq norm}
\|(a+\ui\cdot + A)^{-1}\|_{\Ma_{p,q}(\R;\calL(Y,X))}\leq g(1/a).
\end{equation}
Then $\|T(t)\|_{\calL(Y,X)} \leq C_{q}(g(t)+1)$ for all $t>0$. Here $C_{q} = e C_T C_{Y} M_{\omega} (1+2M\omega)$ for $q<\infty$, $C_{\infty}=eC_T C_Y M_{\omega} (1+ \omega)$, and $C_{Y}=\max(1,\|I_{Y}\|_{\La(Y,X)})$.
\end{theorem}
\begin{proof}
Set $m_{a}(\xi) := (a+i\xi  +A)^{-1}\in\La(Y,X)$ for $a>0$ and $\xi\in\R$. We first prove
\begin{equation}\label{eq:infty bound}
\|m_a\|_{\Ma_{p,\infty}(\R;\La(Y,X))}\leq 2M(g(1/a)+C_{Y})
\end{equation}
for $q<\infty$. Let $f\in \Schw(\R)\otimes Y_0$ be such that $\|f\|_{L^{p}(\R;Y)}\leq 1$.
Then $\|T_{m_{a}}(f)\|_{L^{q}(\R;X)}\leq g(1/a)$, so for each $l\in\Z$ there exists a $t\in[l,l+1]$ such that
\begin{equation}\label{eq:pointwise bound}
\|T_{m_a}(f)(t)\|_{X}\leq 2g(1/a).
\end{equation}
Fix an $l\in\Z$ and let $t\in [l,l+1]$ be such that \eqref{eq:pointwise bound} holds.
Let $\tau\in[0,2]$ and note that (see \cite[Lemma II.1.9]{Engel-Nagel00})
\[
\ue^{-\ui\xi\tau}e^{-a\tau} T(\tau)(a+i\xi +A)^{-1}x=(a+i\xi +A)^{-1}x-\int_{0}^{\tau}\ue^{-(a+\ui\xi)r}T(r)x\ud r
\]
for all $\xi\in\R$ and $x\in X$. Hence
\begin{align*}
e^{-a\tau}T(\tau)T_{m_a}(f)(t)&=\frac{1}{2\pi}\!\int_{\R}\ue^{\ui\xi (t+\tau)}\ue^{-\ui\xi\tau}e^{-a\tau} T(\tau) (a+i\xi +A)^{-1}  \widehat{f}(\xi)\ud\xi\\
&=T_{m_a}(f)(t+\tau)-\!\int_{0}^{\tau}e^{-ar}T(r)f(t+\tau-r)\ud r.
\end{align*}
Rearranging terms and using \eqref{eq:pointwise bound} and H\"{o}lder's inequality, we obtain
\[
\|T_{m_a}(f)(t+\tau)\|_{X}
\leq 2M g(1/a) + \tau^{1/p'}\!M C_{Y} \leq 2M\big(g(1/a)+C_{Y}\big).
\]
Because $\tau\in[0,2]$ and $l\in\Z$ are arbitrary and since $Y_{0}\subseteq Y$ is dense, \eqref{eq:infty bound} follows. This in turn yields
\begin{equation}\label{eq:boundedness multiplier from intersection n larger than one}
\|T_{I_{Y}+\w m_a} (f)\|_{L^{\infty}(\R;X)} \leq C_{Y}\|f\|_{L^{\infty}(\R;Y)} + 2 M\w(g(1/a)+C_{Y}) \|f\|_{L^{p}(\R;Y)}
\end{equation}
for $f\in L^\infty(\R;Y_{0})\cap L^p(\R;Y_{0})$. On the other hand, for $q=\infty$ one has
\begin{equation}\label{eq:boundednessmultiplier2}
\|T_{I_{Y}+\w m_a} (f)\|_{L^{\infty}(\R;X)} \leq C_{Y}\|f\|_{L^{\infty}(\R;Y)}+\w g(1/a)\|f\|_{L^{p}(\R;Y)}
\end{equation}
for all piecewise continuous $f\in L^{p}(\R;Y_{0})\cap L^{\infty}(\R;Y_{0})$ that vanish at infinity.

Let $x\in Y_{0}$ and set $f(t):=\ue^{-(\w+a) t}T(t)x$ for $t\geq0$. It follows from $\C_-\subseteq \rho(A)$ and $[t\mapsto e^{-at}T(t)x]\in L^1([0,\infty);X)$ that (see \cite[Lemma 3.1]{Rozendaal-Veraar18b})
\begin{equation}\label{eq:Fouriertne}
\F ([t\mapsto e^{-at}T(t)x])(\cdot) = (a+\ui \cdot+A)^{-1} x\quad\text{and}\quad \F (f)(\cdot) = (a+\w + \ui \cdot+A)^{-1} x.
\end{equation}
For $t>0$ one has, by the assumptions on $Y$,
\[
\|f(t)\|_Y\leq C_{T}\|\ue^{-(\w+a) t}T(t)\|_{\La(X)}\|x\|_{Y}
\leq C_T M_{\omega} e^{-t}\|x\|_Y.
\]
Hence $f$ is piecewise continuous, vanishes at infinity, and satisfies $\|f\|_{L^{r}(\R_{+};Y)} \leq C_T M_{\omega} \|x\|_{Y}$ for $r\in \{p,\infty\}$.
Also, by \eqref{eq:Fouriertne} and the resolvent identity,
\[
\ue^{-at}T(t)x=T_{I_{Y}+\w m_{a}}(f)(t).
\]
Now \eqref{eq:boundedness multiplier from intersection n larger than one} yields
\[
e^{-at}\|T(t)x\|_X \leq C_T C_Y M_{\omega} (1+2M\omega) (g(1/a)+1) \|x\|_{Y},
\]
and \eqref{eq:boundednessmultiplier2} implies
\[
e^{-at}\|T(t)x\|_X \leq C_T C_{Y}M_{\omega}(1+\omega)(g(1/a)+1) \|x\|_{Y}.
\]
Since $Y_0\subseteq Y$ is dense, the proof is concluded by setting $a=1/t$.
\end{proof}

\begin{remark}\label{rem:large t}
Note from the proof of Theorem \ref{thm:abstract polynomial growth} that if there exist $a_{0}\in(0,\infty)$, $p,q\in[1,\infty]$, and a nondecreasing $g:(0,\infty)\to(0,\infty)$ such that \eqref{eq:pq norm} holds for all $a\in(0,a_{0})$, then $\|T(t)\|_{\La(Y,X)}\leq C (g(t)+1)$ for all $t>1/a_0$. This will be used in the proof of Theorem \ref{thm:analyticmain}.
\end{remark}

\subsection{Hilbert spaces}

We apply Theorem \ref{thm:abstract polynomial growth} by bounding the $\Ma_{p,q}$ norm in \eqref{eq:pq norm} by a supremum norm of $(a+\ui\cdot+A)^{-1}$. We first consider the Hilbert space setting, where the following theorem, in the special case where $g$ is a polynomial, improves \cite[Corollary 2.2]{Eisner-Zwart07}. More general $g$ were considered in \cite[Theorem 3.4]{Boukdir15}, where a bound of the form $\|T(t)\|_{\La(X)}\leq \frac{Cg(t)^2}{t}$ was obtained. Note that $g$ which grow sublinearly lead to exponentially stable semigroups.

\begin{theorem}\label{thm:polynomialgrowthHS}
Let $-A$ be the generator of a $C_{0}$-semigroup $(T(t))_{t\geq0}$ on a Hilbert space $X$ such that $\C_-\subseteq \rho(A)$. Suppose that there exists a nondecreasing $g:(0,\infty)\to (0,\infty)$ such that
\begin{equation}\label{eq:Hilbert space main}
\|(\lambda+ A)^{-1}\|_{\La(X)} \leq g(\Real(\lambda)^{-1})\qquad (\lambda\in\C_{+}).
\end{equation}
Then $\|T(t)\|_{\La(X)} \leq e M_{\omega} (1+2M\omega) (g(t)+1)$ for all $t>0$.
\end{theorem}
\begin{proof}
Condition \eqref{it:condabstract polynomial growth2} in Theorem \ref{thm:abstract polynomial growth}, with $Y_{0}=X_{2}$ and $Y = X$, is satisfied by Proposition \ref{prop:polynomialgrowthgeneral}.  Moreover, Plancherel's identity yields
\[
\|(a+\ui\cdot + A)^{-1}\|_{\Ma_{2,2}(\R;\calL(X))} = \|(a+\ui\cdot + A)^{-1}\|_{L^\infty(\R;\calL(X))} \leq g(1/a),
\]
so that Theorem \ref{thm:abstract polynomial growth} concludes the proof.
\end{proof}

The following example, an extension of an example from \cite{Eisner-Zwart06}, shows that for $g$ a polynomial, Theorem \ref{thm:polynomialgrowthHS} is optimal up to arbitrarily small polynomial loss.

\begin{example}\label{ex:Hilbert}
Fix $\gamma\in (0,1)$ and $n\in\N$. It is shown in \cite{Eisner-Zwart06} that there exist a Hilbert space $X$, a $C_{0}$-semigroup $(S(t))_{t\geq0}\subseteq\La(X)$ with bounded generator $-A$, and constants $C_{1},C_{2}\geq0$ such that $\sigma(A)\subseteq\overline{\C_{+}}$,
\[
\|R(\lambda,A)\|_{\La(X)}\leq\frac{C_{1}}{\Real(\lambda)}\qquad(\lambda\in \C_{-})
\]
and $\|S(t)\|_{\La(X)}\geq C_{2}(t^{\gamma}+1)$ for all $t\geq0$. Let $J\in\La(X^{n})$ be the $n\times n$ operator matrix with $J_{k, k+1} = -I_{X}$ for $k\in \{1, \ldots, n-1\}$, and $J_{k,l}=0$ for $l\neq k+1$. Set $\mathcal{A} := A(I_{X^{n}}+J)$, and let $(T(t))_{t\geq0}\subseteq\La(X^{n})$ be the $C_{0}$-semigroup generated by $-\mathcal{A}$. Then $T(t) = S(t) e^{-tJ}$ for $t\geq0$, and $\|T(t)\|_{\La(X^{n})}\geq c(t^{\gamma+n-1}+1)$ for some $c>0$ independent of $t$. Moreover, there exists a $C\geq0$ such that $\|(\lambda+\mathcal{A})^{-1}\|_{\La(X^{n})}\leq C (\Real(\lambda)^{-n} + 1)$ for all $\lambda\in \C_+$.
\end{example}

\subsection{Asymptotically analytic semigroups}\label{subsec:analytic}

For a $C_{0}$-semigroup $(T(t))_{t\geq0}$ with generator $-A$ on a Banach space $X$, the \emph{non-analytic growth bound} is
\[
\zeta (T)  := \inf \Big\{ \omega\in\R \Big|  \sup_{t>0}e^{-\omega t}\| T(t)-S(t)\|<\infty\;\mbox{for some $S\in\mathcal{H}(\B(X))$}\Big\},
\]
where $\mathcal{H}(\B(X))$ is the set of $S\colon(0,\infty)\to\B(X)$ having an exponentially bounded analytic extension to some sector containing $(0,\infty)$. Let $s_{0}^{\infty}(-A)$  be the infimum over all $\w\in\R$ for which there exists an $R\in(0,\infty)$ such that $\{\eta+\ui\xi\mid \eta>\omega,\xi\in\R,\abs{\xi}\geq R\}\subseteq\rho(-A)$ and
\[
\sup\{\|(\eta+\ui\xi+A)^{-1}\|_{\La(X)}\mid \eta>\omega,\xi\in\R,\abs{\xi}\geq R\}<\infty.
\]
If $\zeta(T)<0$ then $(T(t))_{t\geq 0}$ is \emph{asymptotically analytic}. Then $s_{0}^{\infty}(-A)<0$, and the converse implication holds if $X$ is a Hilbert space. It is trivial that if $(T(t))_{t\geq 0}$ is an analytic semigroup then $\zeta(T)=-\infty$. In fact, $\zeta(T)=-\infty$ if $(T(t))_{t\geq 0}$ is eventually differentiable. For more on asymptotically analytic semigroups see \cite{Blake99,BaBlSr03,Batty-Srivastava03}.

\begin{theorem}\label{thm:analyticmain}
Let $-A$ be the generator of an asymptotically analytic $C_{0}$-semigroup $(T(t))_{t\geq0}$ on a Banach space $X$ such that $\C_-\subseteq \rho(A)$. Suppose that there exists a nondecreasing $g:(0,\infty)\to (0,\infty)$ such that
\[
\|(\lambda+ A)^{-1}\|_{\La(X)} \leq g(\Real(\lambda)^{-1})\qquad (\lambda\in\C_{+}).
\]
Then there exists a $C\geq0$ such that $\|T(t)\|_{\La(X)} \leq C (g(t)+1)$ for all $t>0$.
\end{theorem}
\begin{proof}
By \cite[Theorem 3.6 and Lemmas 3.2 and 3.5]{Batty-Srivastava03} there exist $a_{0}>0$ and $\psi\in C_{c}^{\infty}(\R)$ such that $(1-\psi(\cdot))(a+\ui\cdot+A)^{-1}\in \Ma_{1,\infty}(\R;\La(X))$ for all $a\in(0,a_{0})$, with
\[
C_{1}:=\sup\{\|(1-\psi(\cdot))(a+\ui\cdot+A)^{-1}\|_{\Ma_{1,\infty}(\R;\La(X))}\mid a\in(0, a_{0})\}<\infty.
\]
On the other hand, a straightforward estimate (see also \cite[Proposition 3.1]{Rozendaal-Veraar18b}) shows that $\psi(\cdot)(a+i\cdot+A)^{-1}\in\Ma_{1,\infty}(\R;\La(X))$ for all $a>0$, with
\[
\|\psi(\cdot)(a+i\cdot+A)^{-1}\|_{\Ma_{1,\infty}(\R;\La(X))}\leq \frac{1}{2\pi}\|\psi(\cdot)(a+i\cdot+A)^{-1}\|_{L^{1}(\R;\La(X))}\leq C_{2}g(1/a)
\]
for some $C_{2}\geq0$ independent of $a$. It follows that
\[
\|(a+i\cdot+A)^{-1}\|_{\Ma_{1,\infty}(\R;\La(X))}\leq C_{1}+\frac{R}{2\pi}g(1/a)\leq C_{3}g(1/a)\qquad(a\in(0, a_{0})),
\]
where $C_{3}=C_{1}g(1/a_{0})^{-1}+C_{2}$. Then Remark \ref{rem:large t} yields a constant $C'\geq0$ such that $\|T(t)\|_{\La(X)}\leq C'(g(t)+1)$ for all $t>1/a_{0}$. Since $\sup\{\|T(t)\|_{\La(X)}\mid t\in[0,1/a_{0}]\}<\infty$, this concludes the proof.
\end{proof}

\subsection{Positive semigroups\label{subs:possemi}}

We now consider positive $C_{0}$-semigroups on various Banach lattices. To this end we first prove a multiplier theorem for positive kernels. Part of this result is already contained in \cite[Theorem 3.24]{Rozendaal-Veraar18}. Recall that a subspace $X$ of a Banach lattice $Y$ is a \emph{sublattice} if $x\vee y,x\wedge y\in X$ for all $x,y\in X$.

\begin{proposition}\label{prop:positivekernel}
Let $n\in \N$, $p\in[1,\infty]$, and let $X$ be a Banach lattice and $m:\Rn\to\La(X)$ an $X$-strongly measurable map of moderate growth. Let $K:\Rn\to\La(X)$ be such that $K(\cdot)x\in L^{1}(\Rn;X)$ and $m(\xi)x=\F(K(\cdot)x)(\xi)$ for all $x\in X$ and $\xi\in\Rn$, and such that $K(t)$ is a positive operator for all $t\in\Rn$. Suppose that one of the following conditions holds:
\begin{enumerate}
\item\label{it:Lp} $X=L^{p}(\Omega)$ for $\Omega$ a measure space;
\item\label{it:Cb} $p=\infty$ and $X$ is a closed subspace of $C_{b}(\Omega)$, for $\Omega$ a topological space, such that either $\mathbf{1}_{\Omega}\in X$ or $X$ is a sublattice.
\end{enumerate}
Then $m\in\Ma_{p,p}(\Rn;\La(X))$ with
\[
\|m\|_{\Ma_{p,p}(\Rn;\La(X))}= \|m(0)\|_{\La(X)}.
\]
\end{proposition}
\begin{proof}
It is well known that
\[
\|m\|_{\Ma_{p,p}(\Rn;\La(X))}\geq \sup_{\xi\in\Rn}\|m(\xi)\|_{\La(X)}\geq \|m(0)\|_{\La(X)}
\]
if $m\in\Ma_{p,p}(\Rn;\La(X))$. In the case where $X=L^{p}(\Omega)$ for $p\in[1,\infty)$ it follows from the proof of \cite[Theorem 3.24]{Rozendaal-Veraar18} or \cite[Theorem 2]{Weis98} that $m\in\Ma_{p,p}(\R;\La(X))$ with the required estimate.

Next, assume that $p=\infty$ and let $f:=\sum_{k=1}^{m}\one_{E_{k}}\otimes x_{k}$ for $m\in\N$, $E_{1},\ldots, E_{n}\subseteq \Rn$ disjoint and measurable, and $x_{1},\ldots, x_{n}\in X$. If $\mathbf{1}_{\Omega}\in X$ set $g\equiv\|f\|_{L^{\infty}(\Rn;X)}$, and for $X$ a sublattice set $g=\vee_{1\leq k\leq m}|x_{k}|$. In both cases $g\in X$, $|f(t)|\leq g$ for all $t\in\Rn$, and $\|f\|_{L^{\infty}(\Rn;X)}=\|g\|_{X}$. Then
\[
|T_{m}(f)(t)|\leq \int_{\Rn}|K(s)f(t-s)|\ud s\leq \int_{\Rn}K(s)g\ud s=m(0)g
\]
for all $t\in\Rn$. Hence
\[
\|T_{m}(f)\|_{L^{\infty}(\Rn;X)}\leq \|m(0)\|_{\La(X)}\|g\|_{X}=\|m(0)\|_{\La(X)}\|f\|_{L^{\infty}(\Rn;X)},
\]
which concludes the proof.
\end{proof}

We now prove our main result for positive semigroups.

\begin{theorem}\label{thm:positivepolygrowth}
Let $-A$ be the generator of a positive $C_{0}$-semigroup $(T(t))_{t\geq0}$ on a Banach lattice $X$ such that $\C_-\subseteq\rho(A)$. Assume that one of the following conditions holds:
\begin{enumerate}
\item\label{it:posLp} $X=L^{p}(\Omega)$ for $p\in[1,\infty]$ and $\Omega$ a measure space;
\item\label{it:posC} $p=\infty$ and $X$ is a closed subspace of $C_{b}(\Omega)$, for $\Omega$ a topological space, such that either $\mathbf{1}_{\Omega}\in X$ or $X$ is a sublattice.
\end{enumerate}
Suppose that there exists a nondecreasing $g:(0,\infty)\to (0,\infty)$ such that
\begin{equation}\label{eq:resolventLp}
\|(a+ A)^{-1}\|_{\calL(X)}\leq g(1/a)\qquad(a\in(0,\infty)).
\end{equation}
Then $\|T(t)\|_{\La(X)} \leq C(g(t)+1)$ for all $t>0$, where $C=e M_{\omega}(1+2M\omega)$ for \eqref{it:posLp}, and $C=eM_{\w}(1+\w)$ if \eqref{it:posC} holds.
\end{theorem}
\begin{proof}
Set $p=\infty$ if \eqref{it:posC} holds. Let $a>0$. We first claim that $[t\mapsto \ue^{-at}T(t)x]\in L^{1}([0,\infty);X)$ for all $x\in X$, with
\[
\F([t\mapsto \ue^{-at}T(t)x])(\xi)=(a+i\xi+A)^{-1}x\qquad (\xi\in\R).
\]
To prove this let $n\geq2\w$ and $b\in(0,\min(a,\w))$, and set $B_{n}:=n^{2}(n+A)^{-2}$ and $K_{n,b}(t):=\ue^{-bt}T(t)B_{n}$ for $t\geq0$. Then $K_{n,b}(t)$ is a positive operator for all $t\geq0$, and $K_{n,b}(\cdot)x\in L^{1}(\R;X)$ with
\[
\F(K_{n,b}(\cdot)x)(\xi)=(b+\ui\xi+A)^{-1}B_{n}x\qquad (\xi\in\R),
\]
where we use Proposition \ref{prop:polynomialgrowthgeneral}. By Proposition \ref{prop:positivekernel}, $(b+i\cdot+A)^{-1}B_{n}\in \Ma_{p,p}(\R;\La(X))$ with
\begin{equation}\label{eq:multipliermodified}
\|(b+i\cdot+A)^{-1}B_{n}\|_{\Ma_{p,p}(\R;\La(X))}\leq 4 g(1/b)M_{\w}^{2},
\end{equation}
where we used \eqref{eq:assTM} to deduce that $\|n(n+A)^{-1}\|_{\La(X)}\leq \frac{n}{n-\omega+1} M_{\omega}\leq 2M_{\omega}$.
Let $x\in X$ and set $f(t):=\ue^{-\w t}T(t)x$ for $t\geq0$. Then $f\in L^{p}(\R;X)\cap L^{1}(\R;X)$ is piecewise continuous and vanishes at infinity, and $K_{n,b}\ast f=T_{(b+i\cdot+A)^{-1}B_{n}}(f)$. Moreover,
\[
K_{n,b}\ast f(t)=\int_{0}^{t}\ue^{-(\w-b)s}\ue^{-bt}T(t)B_{n}x\ud s =\frac{1-\ue^{-(\w-b)t}}{\w-b}\ue^{-bt}T(t)B_{n}x.
\]
Since $B_{n}\to I_{X}$ strongly as $n\to\infty$, \eqref{eq:multipliermodified} yields a constant $C_{b}\geq0$ such that
\[
e^{-bt}\|T(t)x\|_{X}\leq C_{b}\|x\|_{X}\qquad(t\geq1).
\]
This shows that $[t\mapsto \ue^{-at}T(t)x]\in L^{1}([0,\infty);X)$ for all $x\in X$, and the identity 
\[
\F([t\mapsto \ue^{-at}T(t)x])(\xi)=(a+i\xi+A)^{-1}x\qquad (\xi\in\R)
\]
is then straightforward. This proves the claim.

Finally, since $\ue^{-at}T(t)$ is a positive operator for all $t\geq0$, Proposition \ref{prop:positivekernel} yields $(a+\ui\cdot+A)^{-1}\in \Ma_{p,p}(\R;\La(X))$ with
\[
\|(a+i\cdot+A)^{-1}\|_{\Ma_{p,p}(\R;\La(X))}=\|(a+A)^{-1}\|_{\La(X)}\leq g(1/a).
\]
Now Theorem \ref{thm:abstract polynomial growth} concludes the proof.
\end{proof}

Theorem \ref{thm:positivepolygrowth} implies in particular that $\w_{0}(T) = s(-A)$ for a positive semigroup $(T(t))_{t\geq0}$ on a space $X$ as in \eqref{it:posLp} or \eqref{it:posC}. For \eqref{it:posLp} this result was originally obtained in \cite{Weis95}.
It is possible to extend Theorem \ref{thm:positivepolygrowth} to fractional domains on more general Banach lattices, by using Fourier multipliers on $X$-valued Besov spaces as in \cite[Theorem 5.7]{Rozendaal-Veraar18b}, but we will not pursue this matter here.

We do not know whether the growth rate in Theorem \ref{thm:positivepolygrowth} is optimal. It follows from \cite[Example 4.4]{Weis97} that the positivity assumption cannot be dropped in case \eqref{it:Lp} for $p\neq 2$. Moreover, \cite[Example 5.1.11]{ArBaHiNe11}) shows that Theorem \ref{thm:positivepolygrowth} is not valid on $X = L^{p}(\Omega)\cap L^{q}(\Omega)$ for $\Omega$ a measure space and $p,q\in [1, \infty)$ with $p\neq q$.

\subsection{Fourier and Rademacher type}

We now improve Proposition \ref{prop:polynomialgrowthgeneral} under additional geometric assumptions on $X$.
A Banach space $X$ is said to have {\em Fourier type $p\in [1, 2]$} if the Fourier transform $\F$ is bounded from $L^p(\R;X)$ into $L^{p'}(\R;X)$. See \cite{HyNeVeWe16} for more on Fourier type. Note in particular that $L^{u}(\Omega)$, for $\Omega$ a measure space and $u\in[1,\infty]$, has Fourier type $p=\min(u,u')$.

\begin{proposition}\label{prop:polynomialgrowthFouriertype}
Let $-A$ be the generator of a $C_{0}$-semigroup $(T(t))_{t\geq0}$ on a Banach space $X$ with Fourier type $p\in[1,2]$ such that $\C_-\subseteq\rho(A)$. Suppose that there exists a nondecreasing $g:(0,\infty)\to (0,\infty)$ such that
\[
\|(\lambda+ A)^{-1}\|_{\La(X)} \leq g(\Real(\lambda)^{-1})\qquad (\lambda\in\C_{+}).
\]
Then for each $\gamma\in(\frac{1}{p}-\frac{1}{p'},\infty)$ there exists a $C_{\gamma}\geq0$ such that $\|T(t)\|_{\La(X_{\gamma},X)} \leq C_{\gamma} (g(t)+1)$ for all $t>0$. For $p=2$ one may let $\gamma=0$.
\end{proposition}
\begin{proof}
The case where $p=1$ follows from Proposition \ref{prop:polynomialgrowthgeneral}.
Hence we may suppose that $\gamma\in[0,1)$, and we may also assume that $g(s)>c$ for all $s>0$ and some $c>0$.  Then \eqref{eq:assTM} yields
\[
\sup_{\lambda>2\w}\lambda\|(\lambda+A+a)^{-1}\|_{\La(X)}\leq 2M_{\w}\leq 2 c^{-1}M_{\w}g(1/a)\qquad(a>0).
\]
Hence $A+a$ is an injective sectorial operator, and for $\theta\in(0,\pi)$ large enough there exists a $C_{1}\geq0$ independent of $a$ such that
\[
\sup_{\lambda\notin\overline{\Se_{\theta}}}\|\lambda R(\lambda,A+a)\|_{\La(X)}\leq C_{1}\sup_{\lambda>0}\|\lambda(\lambda+A+a)^{-1}\|_{\La(X)}\leq 2C_{1}(c^{-1}M_{\w}+\w)g(1/a),
\]
by \eqref{eq:sectorial}. It now follows from the proof of \cite[Proposition 3.4]{Rozendaal-Veraar18b} applied to the operator $A+a$, by keeping track of the relevant constants, that
\[
\|(a+\ui\xi + A)^{-1}\|_{\calL(X_{\gamma},X)}\leq C_{2} (1+|\xi|)^{-\gamma} g(1/a)\qquad(\xi\in\R)
\]
for some $C_{2}\geq0$. Hence \cite[Proposition 3.9]{Rozendaal-Veraar18} yields constants $C_{3},C_{4}\geq0$ such that, for $r\in[1,\infty]$ such that $\frac1r =  \frac1p-\frac1{p'}$ (here one can allow $\gamma=\frac{1}{p}-\frac{1}{p'}=0$ for $p=2$),
\[
\|(a+\ui\cdot + A)^{-1}\|_{\Ma_{p,p'}(\R;\calL(X_{\gamma},X))} \leq C_{3} \|(a+\ui\cdot + A)^{-1}\|_{L^r(\R;\calL(X_{\gamma},X))} \leq C_{4} g(1/a).
\]
Now let $Y:=X_{\gamma}$ and $Y_{0}:=X_{2}$ in Theorem \ref{thm:abstract polynomial growth}, using Proposition \ref{prop:polynomialgrowthgeneral}.
\end{proof}

A similar result holds under type and cotype assumptions on the underlying space, and $R$-boundedness assumptions on the resolvent. Let $(r_{k})_{k\in\N}$ be a sequence of independent real Rademacher variables on some probability space. Let $X$ and $Y$ be Banach spaces and $\mathcal{T}\subseteq \La(X,Y)$. We say that $\mathcal{T}$ is \emph{$R$-bounded} if there exists a constant $C\geq0$ such that for all $n\in\N$, $T_{1},\ldots, T_{n}\in\mathcal{T}$ and $x_{1},\ldots, x_{n}\in X$ one has
\[
\Big(\mathbb{E}\Big\|\sum_{k=1}^{n}r_{k}T_{k}x_{k}\Big\|_{Y}^{2}\Big)^{1/2}\leq
C\Big(\mathbb{E}\Big\|\sum_{k=1}^{n}r_{k}x_{k}\Big\|_{X}^{2}\Big)^{1/2}\!.
\]
The smallest such $C$ is the \emph{$R$-bound} of $\mathcal{T}$ and is denoted by $R(\mathcal{T})$. When we want to specify the underlying spaces $X$ and $Y$ we write $R_{X,Y}(\mathcal{T})$ for the $R$-bound of $\mathcal{T}$, and we write $R_{X}(\mathcal{T}):=R_{X,Y}(\mathcal{T})$ if $X=Y$.

For the definitions of and background on type and cotype we refer to \cite{DiJaTo95,HyNeVeWe2}, and for $p$-convexity and $q$-concavity of Banach lattices see \cite{Lindenstrauss-Tzafriri79}. Note that $X=L^{u}(\Omega)$, for $u\in[1,\infty)$ and $\Omega$ a measure space, has type $p=\min(u,2)$ and cotype $q=\max(2,u)$ and is $u$-convex and $u$-concave. For such $X$ the first statement of the following proposition yields the same conclusion as Proposition \ref{prop:polynomialgrowthFouriertype}.

\begin{proposition}\label{prop:polynomialgrowthcotype}
Let $-A$ be the generator of a $C_{0}$-semigroup $(T(t))_{t\geq0}$ on a Banach space $X$ with type $p\in[1,2]$ and cotype $q\in[2,\infty)$ such that $\C_-\subseteq\rho(A)$. Suppose that there exists a nondecreasing $g:(0,\infty)\to (0,\infty)$ such that
\[
\|(\lambda+ A)^{-1}\|_{\La(X)} \leq g(\Real(\lambda)^{-1})\qquad (\lambda\in\C_{+}).
\]
Then for each $\gamma\in(\frac{2}{p}-\frac{2}{q},\infty)$ there exists a $C_{\gamma}\geq0$ such that $\|T(t)\|_{\La(X_{\gamma},X)} \leq C_{\gamma} (g(t)+1)$ for all $t> 0$. If
\[
R_{X}(\{(a+\ui\xi + A)^{-1}\mid \xi\in\R\})\leq g(1/a)\qquad (a\in(0,\infty)),
\]
then one may let $\gamma\in(\frac{1}{p}-\frac{1}{q},\infty)$. If in addition $X$ is a $p$-convex and $q$-concave Banach lattice then one may let $\gamma=\frac{1}{p}-\frac{1}{q}$.
\end{proposition}

One could also let $q=\infty$ in the first two statements in this proposition. However, then Proposition \ref{prop:polynomialgrowthgeneral} yields a stronger statement, since any Banach space has type $p=1$ and cotype $q=\infty$, and because a Banach space that does not have finite cotype also does not have nontrivial type.

\begin{proof}
We may suppose that $\gamma\in(0,1)$, by Proposition \ref{prop:polynomialgrowthgeneral} and because each $2$-convex and $2$-concave Banach lattice is isomorphic to a Hilbert space, by \cite{Kwapien72}. We may also suppose that $g(s)>c$ for all $s>0$ and some $c>0$. We first prove the final two statements. 

As in the proof of Proposition \ref{prop:polynomialgrowthFouriertype}, it suffices to check the multiplier condition in Theorem \ref{thm:abstract polynomial growth}. Moreover, again using estimates in the proof of \cite[Proposition 3.4]{Rozendaal-Veraar18b} and proceeding as in the proof of Proposition \ref{prop:polynomialgrowthFouriertype}, one obtains a $C_{1}\geq0$ such that
\[
R_{X_{\gamma},X}(\{(1+|\xi|)^{\gamma} (a+\ui\xi + A)^{-1}\mid \xi\in \R\})\leq C_{1} g(1/a)\qquad(a>0).
\]
Now \cite[Theorems 3.18 and 3.21]{Rozendaal-Veraar18} yield a $C_{2}\geq0$ such that
\[
\|(a+\ui\cdot + A)^{-1}\|_{\Ma_{p,q}(\R;\calL(X_{\gamma},X))} \leq C_{2} g(1/a)\qquad(a>0),
\]
which proves the final two statements.

For the first statement we may assume that $\frac{2}{p}-\frac{2}{q}<1$ and show that for each $\gamma\in(\frac{2}{p}-\frac{2}{q},1)$ there exists a $C_{3}\geq0$ such that
\begin{equation}\label{eq:Rbound}
R_{X_{\gamma},X}(\{(1+|\xi|)^{\gamma/2} (a+\ui\xi + A)^{-1}\mid \xi\in \R\})\leq C_{3} g(1/a)\qquad(a>0),
\end{equation}
after which one proceeds as before. To obtain \eqref{eq:Rbound} let $r\in[1,\infty]$ be such that $\frac{1}{r}=\frac{1}{p}-\frac{1}{q}$, and set $f_{a}(\xi):=(1+|\xi|)^{\gamma/2}(a+\ui\xi+A)^{-1}$ for $\xi\in\R$. Then $f_{a}\in W^{1,r}(\R;\La(X_{\gamma},X))$ by \cite[Proposition 3.4]{Rozendaal-Veraar18b}, with
\[
\|f_{a}\|_{W^{1,r}(\R;\La(X_{\gamma},X))}\leq C_{4}g(1/a)
\]
for some $C_{4}\geq0$ independent of $a$. Now \cite[Lemma 2.1]{Rozendaal-Veraar18b} yields \eqref{eq:Rbound}.
\end{proof}

It follows from an example due to Arendt (see \cite[Example 5.1.11]{ArBaHiNe11} or \cite[Section 4]{Weis-Wrobel96}) that, already in the case where $g$ is constant, the indices $\frac{1}{p}-\frac{1}{p'}$ and $\frac{1}{p}-\frac{1}{q}$ in Propositions \ref{prop:polynomialgrowthFouriertype} and \ref{prop:polynomialgrowthcotype} cannot be improved. We do not know whether it is in general possible to let $\gamma=\frac{1}{p}-\frac{1}{p'}$ or $\gamma=\frac{1}{p}-\frac{1}{q}$ in these results.

\subsection{Necessary conditions}

Here we provide a converse to Theorem \ref{thm:abstract polynomial growth}, extending \cite[Theorem 2.1]{Eisner-Zwart07}. For simplicity we restrict to semigroups of polynomial growth and to fractional domains, but from the proof one can derive an analogous statement for more general semigroups and more general continuously embedded spaces.

\begin{theorem}\label{thm:abstractconverse}
Let $-A$ be the generator of a $C_{0}$-semigroup $(T(t))_{t\geq0}$ on a Banach space $X$. Let $\gamma\in[0,\infty)$. Suppose that there exist $\alpha,C\geq0$ such that $\|T(t)\|_{\La(X_{\gamma},X)}\leq C(t^{\alpha}+1)$ for all $t\geq0$. 
Then $\C_{-}\subseteq\rho(A)$ and for all $p\in[1,\infty]$, $q\in[p,\infty]$, and $r\in[1,\infty]$ such that $\frac{1}{p}-\frac{1}{q}=1-\frac{1}{r}$, we have
\begin{align}\label{eq:converse1}
\|(a+i\cdot+A)^{-1}\|_{\Ma_{p,q}(\R;\La(X_{\gamma},X))}\leq C(C_{r}a^{-\alpha-\frac{1}{r}}+C'_{r}a^{-\frac{1}{r}})\quad (a\in(0,\infty)),
\end{align}
where $C_{r}=r^{-\alpha-\frac{1}{r}}\Gamma(\alpha+1)^{\frac{1}{r}}$ and $C'_{r}=r^{-1/r}$ for $r<\infty$, and $C_{\infty}=e^{-\alpha}\alpha^{\alpha}$ and $C'_{\infty}=1$. Moreover,
\begin{equation}\label{eq:converse2}
\begin{aligned}
\sup\{\|(a+i\xi+A)^{-1}\|_{\La(X_{\gamma},X)}\mid \xi\in \R\} & \leq R_{X_{\gamma},X}(\{(a+i\xi+A)^{-1}\mid \xi\in \R\})
\\ & \leq C (\Gamma(\alpha+1)a^{-\alpha-1}+a^{-1}).
\end{aligned}
\end{equation}
\end{theorem}
\begin{proof}
It follows by rescaling from \cite[Proposition 4.19]{Rozendaal-Veraar18b} that $\C_{-}\subseteq\rho(A)$. We claim
\begin{equation}\label{eq:rnormsemigroup}
\|e^{-a\cdot}\|T(\cdot)\|_{\La(X_{\gamma},X)}\|_{L^{r}(0,\infty)}\leq C(C_{r}a^{-\alpha-\frac{1}{r}}+C'_{r}a^{-\frac{1}{r}})\qquad (a\in(0,\infty)).
\end{equation}
To prove this claim, first consider $r<\infty$. Then
\begin{align}
\nonumber &\|e^{-a\cdot}\|T(\cdot)\|_{\La(X_{\gamma},X)}\|_{L^{r}(0,\infty)}\leq C\Big(\int_{0}^{\infty}e^{-art}(t^{\alpha}+1)^{r}\ud t\Big)^{\frac{1}{r}}\\
&
\label{eq:hulprT}
\leq C\Big(\Big(\int_{0}^{\infty}e^{-art}t^{r\alpha}\ud t\Big)^{\frac{1}{r}}\!+\Big(\int_{0}^{\infty}e^{-art}\ud t\Big)^{\frac{1}{r}}\,\Big)\\
&\leq C\Big((ar)^{-\alpha-\frac{1}{r}}\Big(\int_{0}^{\infty}e^{-t}t^{\alpha}\ud t\Big)^{\frac{1}{r}}\!+(ar)^{-\frac{1}{r}}\Big)=C(C_{r}a^{-\alpha-\frac{1}{r}}+C'_{r}a^{-\frac{1}{r}}).
\nonumber
\end{align}
On the other hand, for $r=\infty$ a simple optimization argument shows that
\[
\sup_{t\geq0}e^{-at}\|T(t)\|_{\La(X_{\gamma},X)}\leq C(\sup_{t\geq0}e^{-at}t^{\alpha}+1)=C\big(e^{-\alpha}\alpha^{\alpha}a^{-\alpha}+1\big).
\]

Now set $m_{a}(\xi):=(a+i\xi+A)^{-1}$ for $a>0$ and $\xi\in\R$. For $r<\infty$ let $f\in\Sw(\R)\otimes X$, and for $r=\infty$ let $f$ be an $X$-valued simple function. Note that $e^{-a\cdot}\|T(\cdot)\|_{\La(X_{\gamma},X)}\in L^{1}(\R)$. It then follows in a straightforward manner (see \cite[Lemma 3.1]{Rozendaal-Veraar18b}) that
\[
(a+i\xi+A)^{-1}x=\int_{0}^{\infty}e^{-t(a+i\xi)}T(t)x\ud t\qquad(x\in X_{\gamma}, \xi\in\R)
\]
and
\[
T_{m_{a}}(f)=\int_{0}^{\infty}e^{-as}T(s)f(t-s)\ud s\qquad(t\in\R).
\]
The latter equality, \eqref{eq:rnormsemigroup} and Young's inequality for operator-valued kernels \cite[Proposition 1.3.5]{ArBaHiNe11} yield \eqref{eq:converse1}. On the other hand, applying \cite[Corollary 2.17]{Kunstmann-Weis04} and \eqref{eq:hulprT} with $r=1$ to $t\mapsto e^{-at} T(t)$ yields \eqref{eq:converse2}.
\end{proof}

For $-A$ a standard $n\times n$ Jordan block acting on $X=\R^{n}$, $n\geq2$, there exists a $C\geq0$ such that
\[
C^{-1}(t^{n-1}+1)\leq \|T(t)\|_{\La(X)}\leq C(t^{n-1}+1)\qquad(t\geq0)
\]
and
\[
\|(a+\ui\xi+A)^{-1}\|_{\La(X)}\leq \|(a+A)^{-1}\|\leq C (a^{-n} + a^{-1})\qquad (a>0, \xi\in\R).
\]
This shows that \eqref{eq:converse2} is optimal. Note that in this case $R$-boundedness and uniform boundedness coincide since $X$ is a Hilbert space.

\begin{remark}\label{rem:Rbounded}
One might be tempted to think that the more restrictive $R$-bounded analogue of \eqref{eq:R} which appears in \eqref{eq:converse2}, namely
\[
R_{X}(\{(a+i\xi+A)^{-1}\mid \xi\in \R\})\leq g(1/a)\qquad(a\in(0,\infty)),
\]
can be used to extend the conclusion of Theorem \ref{thm:main} to more general Banach spaces. However, the example at the end of Section \ref{subs:possemi} shows that this is not the case for certain positive semigroups on $L^p(\Omega)\cap L^q(\Omega)$, for $\Omega$ a measure space.
\end{remark}

Theorems \ref{thm:abstract polynomial growth} and \ref{thm:abstractconverse} combine to yield the following characterization of polynomially growing semigroups on fractional domains.

\begin{corollary}\label{cor:abstractchar}
Let $-A$ be the generator of a $C_{0}$-semigroup $(T(t))_{t\geq0}$ on a Banach space $X$ such that $\C_{-}\subseteq\rho(A)$, and let $\alpha,\gamma\in[0,\infty)$. Then the following conditions are equivalent:
\begin{enumerate}
\item\label{it:char1} there exists a $C\geq0$ such that $\|T(t)\|_{\La(X_{\gamma},X)}\leq C(t^{\alpha}+1)$ for all $t\geq0$;
\item\label{it:char2} there exist $p,q\in[1,\infty]$ and a $C'\geq0$ such that
\begin{equation}\label{eq:chareq}
\|(a+i\cdot+A)^{-1}\|_{\Ma_{p,q}(\R;\La(X_{\gamma},X))}\leq C'(a^{-\alpha}+1)\qquad(a\in(0,\infty)).
\end{equation}
\end{enumerate}
\end{corollary}
\begin{proof}
Theorem \ref{thm:abstract polynomial growth} contains \eqref{it:char2}$\Rightarrow$\eqref{it:char1}, and \eqref{it:char1}$\Rightarrow$\eqref{it:char2} follows from Theorem \ref{thm:abstractconverse} by letting $p=1$ and $q=\infty$.
\end{proof}

Note that Corollary \ref{cor:abstractchar} also characterizes semigroups which grow sublinearly, and in particular uniformly bounded semigroups. To characterize such semigroups it would not be possible to replace the multiplier norm in \eqref{eq:chareq} by a supremum norm, since $\|R(\lambda,A)\|_{\La(X)}\geq \text{dist}(\lambda,\sigma(A))^{-1}$ for all $\lambda\in\rho(A)$.

\subsection{Auxiliary results}

The theorems in this article also apply if $A$ is an $n\times n$ matrix acting on $X=\Rn$, $n\in\N$. For example, if
\[
\|(a+\ui\xi + A)^{-1}\|_{\La(X)} \leq g(1/a)\qquad (a>0,\xi\in \R)
\]
then one obtains $\|\ue^{-tA}\|_{\La(X)}\leq e M_{\omega}(1+2M\omega) (g(t)+1)$ for all $t>0$ if $\Rn$ is endowed with the standard norm, or if $(\ue^{-tA})_{t\geq0}$ is positive and $\Rn$ is endowed with the $\ell_{p}$-norm, $p\in[1,\infty]$. Here $\omega$, $M$ and $M_{\omega}$ are as in \eqref{eq:assTM}. Note that this estimate does not depend on $n$ but that it does require knowledge of $\omega$, $M$ and $M_{\omega}$. If these constants are unknown then the argument used to prove \cite[Theorem 4.8]{vDoKrSp93} (see also \cite{LeVeque-Trefethen84}) yields the following statement, which is presumably well known to experts. For the convenience of the reader we include the proof. Recall that it suffices to consider the case where $g$ grows at least linearly at infinity and $g(t)=O(t)$ as $t\to0$.

\begin{proposition}\label{prop:matrices}
Let $X$ be an $n$-dimensional normed vector space, $n\in\N$, and let $A\in\La(X)$ be such that $\C_-\subseteq\rho(A)$. Suppose that there exists a nondecreasing $g:(0,\infty)\to (0,\infty)$ such that
\[
\|(a+\ui\xi + A)^{-1}\|_{\La(X)} \leq g(1/a)\qquad (a\in(0,\infty),\xi\in \R).
\]
Then $\|e^{-tA}\|_{\La(X)}\leq en\frac{g(t)}{t}$ for all $t>0$.
\end{proposition}

\begin{proof}
Let $a,t>0$ and write, as in the proof of Proposition \ref{prop:polynomialgrowthgeneral},
\[
\ue^{-at}T(t)=\frac{1}{2\pi\ui}\int_{i\R}e^{-zt}R(z,A+a)\ud z.
\]
Let $F\in \La(X)^{*}$ be such that $\|F\|_{\La(X)^{*}}\leq 1$ and $F(T(t))=\|T(t)\|_{\La(X)}$. Integration by parts yields
\[
\ue^{-at}\|T(t)\|_{\La(X)}=\frac{1}{2\pi\ui}\int_{i\R}e^{-zt}F(R(z,A+a))\ud z=\frac{1}{2\pi\ui t}\int_{i\R}e^{-zt}F(R(z,A+a))'\ud z.
\]
One easily sees that $z\mapsto F(R(z,A+a))$ is a rational scalar-valued map with numerator and denominator of degree at most $n$. Now \cite[Lemma 2]{Spijker91} (after composing with a suitable M\"{o}bius transformation) shows that
\[
\ue^{-at}\|T(t)\|_{\La(X)}\leq \frac{n}{t}\sup_{z\in\ui\R}|F(R(z,A+a))|\leq \frac{ng(1/a)}{t}.
\]
Finally, set $a=1/t$ to conclude the proof.
\end{proof}

Proposition \ref{prop:matrices} is sharp in the case where $g(t)=K t$ for some $K\geq0$ and all $t>0$ (see \cite{vDoKrSp93,LeVeque-Trefethen84,Kraaijev}). For further discussion on this topic we refer the reader to \cite{Nikol}, where in particular improvements on the bounds have been obtained under additional geometric assumptions on the norm of $X$.

Finally, as a corollary of Theorem \ref{thm:analyticmain} we extend a theorem from \cite{Eisner-Zwart08} concerning the growth of the \emph{Cayley transform} $V(A):=(1-A)(1+A)^{-1}$ of a semigroup generator $-A$ on a Banach space $X$ with $-1\in\rho(A)$. Recall from Section \ref{subsec:analytic} that each eventually differentiable semigroup, and in particular each analytic semigroup, is asymptotically analytic. Also, if $-A$ generates a $C_{0}$-semigroup $(T(t))_{t\geq0}$ on a Hilbert space $X$ such that $s_{0}^{\infty}(-A)<0$, then $(T(t))_{t\geq0}$ is asymptotically analytic. Hence the following result both extends and improves \cite[Theorem 5.4]{Eisner-Zwart08}.

\begin{corollary}\label{cor:Cayley}
Let $(T(t))_{t\geq0}$ be an asymptotically analytic $C_{0}$-semigroup with generator $-A$ on a Banach space $X$ such that $-1\in\rho(A)$. Suppose that there exist $k\in \N_{0}$ and $C\geq0$ such that
\[
\|V(A)^{n}\|_{\La(X)}\leq Cn^{k}\qquad(n\in\N).
\]
Then there exists a $C'\geq0$ such that $\|T(t)\|_{\La(X)}\leq C'(1+t^{k+1})$ for all $t\geq0$.
\end{corollary}
\begin{proof}
First note that $s_{0}^{\infty}(-A)<0$, since $(T(t))_{t\geq0}$ is asymptotically analytic (see \cite[Proposition 2.4]{BaBlSr03}). Now proceed as in the proof of \cite[Theorem 5.4]{Eisner-Zwart08} to show that
\[
\|(a+i\xi+A)^{-1}\|_{\La(X)}\leq C_{1}a^{-k-1}\qquad (a>0,\xi\in\R)
\]
for some $C_{1}\geq0$. Theorem \ref{thm:analyticmain} then concludes the proof.
\end{proof}

\section{Application to a perturbed wave equation}\label{sec:exRen}

In \cite{Zabczyk75}, using a direct sum of Jordan blocks, Zabczyk constructed a $C_0$-semigroup $(T(t))_{t\geq0}$ with generator $-A$ on a Hilbert space such that $\omega_0(T)>s(-A)$. One might be tempted to think that this phenomenon only occurs in rather academic situations. However, in \cite[Theorem 1]{Renardy94} Renardy gave an example of a concrete perturbed wave equation with the same property. More precisely, the $C_{0}$-group $(T(t))_{t\in\R}$ with generator $-A$ which arises when formulating this wave equation as an abstract Cauchy problem has the property that $s(-A)=0=s(A)$ but $\w_{0}(T)\geq\frac{1}{2}$. In this section we prove that $\w_{0}(T)=\frac{1}{2}$, a matter which was left open in \cite{Renardy94}. In fact, Theorem \ref{thm:WaveRen} below yields a more precise growth bound for $(T(t))_{t\in\R}$.

On the two-dimensional torus $\T^2:=[0,2\pi]^2$, under the usual identification modulo $2\pi$, consider
\begin{equation}\label{eq:Waveeq}
\left\{
  \begin{array}{ll}
    u_{tt} = u_{xx} + u_{yy} + e^{iy} u_x,  &  t\in(0,\infty), x,y\in \T,  \\
    \\  u(0,x,y)  = f(x,y), \quad  u_t(0,x,y)=g(x,y), &  x,y\in \T,
  \end{array}
\right.
\end{equation}
for $f,g\in L^{2}(\T^{2})$. For $s\in \R$ let $H^{s}(\T^2)=W^{2,s}(\T^{2})$ be the second order Sobolev space equipped with the following convenient norm:
\[
\|f\|_{H^{s}(\T^2)} = \Big(|\wh{f}(0)|^2 + \!\sum_{k\in \Z^2\setminus\{0\}} |k|^{2s} |\wh{f}(k)|^2\Big)^{1/2}\qquad(f\in H^{s}(\T^{2})).
\]
Clearly, this norm is equivalent to the standard norm on $H^{s}(\T^{2})$:
\begin{equation}\label{eq:normequiH}
\|f\|_{H^{s}(\T^2)}\leq \Big(\sum_{k\in \Z^2} (1+|k|^{2})^{s} |\wh{f}(k)|^2\Big)^{1/2}\leq C_{s}\|f\|_{H^{s}(\T^{2})}
\end{equation}
for some $C_{s}\geq0$ and all $f\in H^{s}(\T^{2})$.
Then \eqref{eq:Waveeq} can be formulated as an abstract Cauchy problem on the Hilbert space $X := H^1(\T^2)\times L^2(\T^2)$:
\begin{equation}\label{eq:Aren}
\frac{d}{dt} \left(\begin{matrix}
               u \\
               v
             \end{matrix}\right) + A\left(\begin{matrix}
               u \\
               v
             \end{matrix}\right) = 0
\end{equation}
and $(u(0), v(0)) = (f, g)$, where $A = A_0 + B$ with $D(A) = H^2(\T^2)\times H^1(\T^2)$,
\[ A_0 = \left(
    \begin{array}{cc}
       0 & -1 \\
      -\Delta  & 0 \\
    \end{array}
  \right)\quad \text{and}\quad B = \left(\begin{array}{cc}
       0 & 0 \\
       -M \frac{\partial}{\partial x} & 0 \\
    \end{array}\right).
\]
Here $\Delta$ is the Laplacian with $D(\Delta) = H^2(\T^2)$, and $M:L^2(\T^2)\to L^2(\T^2)$ is given by $Mf(x,y)=e^{iy}f(x,y)$ for $f\in L^{2}(\T^{2})$ and $x,y\in\T$. Using Fourier series one easily checks that $-A_0$ generates a $C_0$-group. More precisely, let $e_k(x,y) := (2\pi)^{-1}e^{i k\cdot (x,y)}$ for $k\in \Z^2$. Taking the discrete Fourier tranform, the system
\[
\frac{d}{dt} \left(\begin{matrix}
               \ph \\
               \psi
             \end{matrix}\right) + A_{0}\!\left(\begin{matrix}
               \ph \\
               \psi
             \end{matrix}\right) = 0
\]
can be solved explicitly. Let $h_{k}:=\frac{1}{2\pi}\int_{\T^{2}}\ue^{-ik\cdot(x,y)}h(x,y)\ud x\ud y$, $k\in\Z^{2}$, be the Fourier coefficients of $h\in L^{2}(\T^{2})$. Then
\begin{align*}
\varphi(t) & = (f_0 + t g_0) e_0 + \sum_{k\in \Z^2\setminus\{0\}} \Big(\cos(|k|t) f_k + \frac{\sin(|k|t)}{|k|} g_k\Big) e_k, \\
\psi(t)& = g_0 e_0 + \sum_{k\in \Z^2\setminus\{0\}} (-|k|\sin(|k|t) f_k + \cos(|k|t) g_k) e_k
\end{align*}
for $t\in\R$. Set $\ue^{-tA_{0}}\!\left(\begin{matrix}
               f \\
               g
             \end{matrix}\right) :=\left(\begin{matrix}
               \ph(t) \\
               \psi(t)
             \end{matrix}\right)$. One has
\begin{align*}
\|(\varphi(t), \psi(t))\|_{X}^2 & = |f_0 + t g_0|^2 + |g_0|^2 +\!\sum_{k\in \Z^2\setminus \{0\}}\!\big(|k|^2 |f_k|^2 + |g_k|^2\big)
\\ & \leq   2|f_0|^2 +\!\sum_{k\in \Z^2\setminus \{0\}}\!|k|^2 |f_k|^2 + 2|t g_0|^2 + |g_0|^2 +\!\sum_{k\in \Z^2\setminus \{0\}} |g_k|^2
\\ & \leq   2\|f\|_{H_1(\T^2)}^2 +  (1+2t^2)\|g\|_{L^2(\T^2)}^2\leq 2(1+|t|)^2\|(f, g)\|_{X}^2,
\end{align*}
so that $\|\ue^{-tA_{0}}\|_{\La(X)}\leq \sqrt{2}(1+|t|)$ for all $t\in\R$. One could alternatively get a norm estimate using Theorem \ref{thm:polynomialgrowthHS}, but in this case one obtains only a quadratic bound.

Since $\|B\|_{\La(X)}\leq 1$, standard perturbation theory (see \cite[Theorem III.1.3]{Engel-Nagel00}) shows that $-A = -A_0 -B$ generates a $C_{0}$-group $(T(t))_{t\in\R}$ with
\begin{equation}\label{eq:semigrouppert}
\|T(t)\|_{\La(X)}\leq \sqrt{2} e^{(1+\sqrt{2})|t|}\qquad(t\in\R).
\end{equation}
It was shown in \cite[Theorem 1]{Renardy94} that $\sigma(A)\subseteq i\R$ and $\omega_0(T)\geq \frac12$, and by the same method one sees that $\omega_0(S)\geq \frac12$ for $(S(t))_{t\geq0}:=(T(t)^{-1})_{t\geq0}$, the semigroup generated by $A$. The next theorem is the main result of this section. It shows that these lower bounds are optimal and in doing so significantly improves \eqref{eq:semigrouppert}.

\begin{theorem}\label{thm:WaveRen}
Let $X$ and $A$ be as before, and let $(T(t))_{t\in\R}$ and $(S(t))_{t\in\R}$ be the $C_{0}$-semigroups generated by $-A$ and $A$, respectively. Then $\omega_0(T) = \omega_0(S)=\frac12$. Moreover, there exists a $C\geq0$ such that
\[
\|T(t)\|_{\La(X)}\leq C(1+|t|)e^{|t|/2}\qquad(t\in \R).
\]
\end{theorem}

\begin{remark}
For each $R\geq0$ there exists a $C_{R}\geq0$ such that $\|(\frac12 + i\xi\pm A)^{-1}\|_{\La(X)}\leq C_R$ for $|\xi|\leq R$, since $\sigma(A) \subseteq i\R$, and it follows from Theorem \ref{thm:WaveRen} that $C_{R}\to\infty$ as $R\to\infty$. It would be interesting to study the asymptotic behavior of $\|(\frac12 + i\xi\pm A)^{-1}\|_{\La(X)}$ as $|\xi|\to \infty$. Moreover, if
$\|e^{-|t|/2} T(t)\|_{\La(X)}$ were to grow asymptotically linearly as $t\to \infty$ then this would solve the optimality issue left open after Theorem \ref{thm:polynomialgrowthHS} and in \cite{Eisner-Zwart06}.
\end{remark}

The proof of Theorem \ref{thm:WaveRen} relies on two lemmas. The first collects some basic estimates.

\begin{lemma}\label{lem:elementarycomplex}
Let $z\in \C$ be such that $|\Real(z)|\geq \frac12$, and let $y\in \R$. Then
\[
(i)\ \frac{|z|^2}{|z^2+y^2|^2}\leq 4, \quad (ii)  \  \frac{y^2+1}{|z^2+y^2|^2}\leq 16,
\quad (iii)
\ \frac{|z|^4}{|z^2+y^2|^2} \leq 32 (y^2+1).
\]
\end{lemma}
\begin{proof}
Write $z = a+is$ for $a, s\in \R$ with $|a|\geq 1/2$. Then (i) and (ii) follow from
\[
|z^2+y^2|^2 = (y^2-s^2)^2 + a^4 + 2y^2 a^2+2a^2 s^2\geq \max(\tfrac{1}{16}(1+y^2), \tfrac14|z|^2).
\]
For (iii) note that
\[
|z|^4\leq (|z^2+y^2| + y^2)^2 \leq 2|z^2+y^2|^2 + 2y^4,
\]
divide by $|z^2+y^2|^2$, and use (ii).
\end{proof}

The following lemma contains the required resolvent estimates for $A$.

\begin{lemma}\label{lem:WaveRen}
Let $X$ and $A$ be as before. Then there exists a $C\geq0$ such that for all $\varepsilon>0$, $\xi\in \R$ and $\lambda = \pm (\frac12 +\varepsilon) + i\xi$ one has
\[
\|(\lambda + A)^{-1}\|_{\La(X)}\leq C \max(\varepsilon^{-1}, 1).
\]
\end{lemma}
\begin{proof}
Let $\lambda\in\C\setminus \ui\R$, $(u,v)\in D(A)$ and $(f,g)\in X$ be such that $(\lambda+ A)(u,v) = (f,g)$. Then
\begin{equation}\label{eq:elliptic}
 \lambda^2 u -\Delta u - e^{iy} u_x = g + \lambda f
\end{equation}
in $L^{2}(\T^{2})$. Since $v =  \lambda u - f$, it suffices to prove
\begin{equation}\label{eq:toproveespilonest}
\|u\|_{H^1(\T^{2})} + \|\lambda u\|_{L^2(\T^{2})} \leq C \max(1,\varepsilon^{-1}) (\|f\|_{H^1(\T^{2})} + \|g\|_{L^2(\T^{2})})
\end{equation}
if $\lambda = \pm(\frac12 +\varepsilon) + i\xi$ for $\varepsilon>0$ and $\xi\in \R$.
Write $u = \sum_{(m,n)\in \Z^2} u_{m,n} e_{m,n}$ with $(u_{m,n})_{m,n\in\Z}$ the Fourier coefficients of $u$ and $(e_{m,n})_{m,n\in\Z}$ the normalized trigonometric basis of $L^{2}(\T^{2})$. Then \eqref{eq:elliptic} yields
\[
(\lambda^2+ m^2+n^2)u_{m,n} =  i m u_{m, n-1} + g_{m,n} + \lambda f_{m,n}\qquad(m,n\in \Z).
\]
Now, using that $|r+s|^2\leq (1+\delta) |r|^2 + (1+\delta^{-1})|s|^2$ for any fixed $\delta>0$ and all $r,s\in \C$, one has
\begin{equation}\label{eq:Fourierest}
|u_{m,n}|^2 \leq  \frac{(1+\delta) |m u_{m, n-1}|^2}{|\lambda^2+ m^2+n^2|^2} + \big(1+\frac1\delta\big) \Big(\frac{|g_{m,n}|}{|\lambda^2+ m^2+n^2|} + \frac{|\lambda f_{m,n}|}{|\lambda^2+ m^2+n^2|}\Big)^2.
\end{equation}

We first bound $\|u\|_{H^1(\T^{2})}$ in \eqref{eq:toproveespilonest}. From \eqref{eq:Fourierest} we obtain
\[
\sum_{m,n\in\Z} (m^2+n^2+1)|u_{m,n}|^2 \leq (1+\delta) \sum_{m,n\in\Z} \frac{m^2 (m^2+(n+1)^2+1)|u_{m, n}|^2}{|\lambda^2+ m^2+(n+1)^2 |^2} + C_{f,g}^2
\]
for
\[
C_{f,g}^2 = \big(1+\frac{1}{\delta}\big) \!\sum_{k\in \Z^2} \Big(\frac{(|k|^2+1)^{1/2} |g_{k}|}{|\lambda^2+ |k|^2|} + \frac{(|k|^2+1)^{1/2}|\lambda f_{k}|}{|\lambda^2+ |k|^2|}\Big)^2.
\]
Lemma \ref{lem:elementarycomplex} (i) and (ii) yield a $C_{1}\geq0$ such that $C_{f,g}\leq C_{1} (1+\delta^{-1})^{1/2} (\|f\|_{H^1}+\|g\|_{L^2})$, so that
\begin{equation}\label{eq:aprioiralmost}
\sum_{m,n\in\Z} (m^2+n^2+1)|u_{m,n}|^2 \big(1-(1+\delta) y_{m,n}\big)
\leq C_{1}^2 (1+\delta^{-1}) (\|f\|_{H^1}+ \|g\|_{L^2})^2
\end{equation}
for
\[
y_{m,n} := \frac{m^2(m^2+(n+1)^2+1)}{(m^2+n^2+1)|\lambda^2+ m^2+(n+1)^2 |^2}\qquad (m,n\in\Z).
\]
Now suppose that $\lambda = a+i\xi$ for $\xi\in \R$ and $|a|> \frac12$. Then a simple minimization argument yields
\begin{equation}\label{eqlambdamin}
|\lambda^2+ m^2+(n+1)^2|^2 = (a^2-\xi^2+m^2+(n+1)^2)^2 + 4 a^2\xi^2 \geq 4a^2 (m^2+(n+1)^2),
\end{equation}
from which it follows that $y_{m,n} \leq \frac{1}{4a^2}$ for all $m,n\in\Z$. Combining this with \eqref{eq:normequiH} and \eqref{eq:aprioiralmost}, we obtain that for $\delta\in(0,4a^{2}-1)$ one has
\[
\|u\|_{H^1(\T^{2})}  \leq   C_{1} \frac{ 2|a|(1+\delta^{-1})^{1/2}}{(4a^2-(1+\delta))^{1/2}} (\|f\|_{H^1(\T^{2})}+ \|g\|_{L^2(\T^{2})}).
\]
For $\varepsilon>0$ such that $|a|=\frac{1}{2}+\varepsilon$ one now easily obtains a $C_{2}\geq0$ independent of $\varepsilon$ such that
\[
\|u\|_{H^1(\T^{2})} \leq  C_{2}\max(1,\varepsilon^{-1}) (\|f\|_{H^1(\T^{2}}+ \|g\|_{L^2(\T^{2})}).
\]

We now bound $\|\lambda u\|_{L^2(\T^{2})}$ in \eqref{eq:toproveespilonest}. From \eqref{eq:Fourierest} one obtains
\begin{equation}\label{eq:aprioriestuRen2}
\sum_{m,n\in\Z} |\lambda|^2|u_{m,n}|^2 \leq (1+\delta) \sum_{m,n} \frac{|\lambda|^{2}m^2 |u_{m, n}|^2}{|\lambda^2+ m^2+(n+1)^2 |^2} + K_{f,g}^2,
\end{equation}
where
\[
K_{f,g}^2 = \big(1+\frac{1}{\delta}\big) \sum_{k\in \Z^2} \Big(\frac{|\lambda| |g_{k}|}{|\lambda^2+ |k|^2|} + \frac{|\lambda|^2 |f_{k}|}{|\lambda^2+ |k|^2|}\Big)^2\!\leq C_{3} (1+\delta^{-1})^{1/2} (\|f\|_{H^1}+\|g\|_{L^2})
\]
for some $C_{3}\geq0$ by Lemma \ref{lem:elementarycomplex} (i) and (iii).
Now \eqref{eq:aprioriestuRen2} implies
\begin{align*}
|\lambda|^2 \sum_{m,n} |u_{m,n}|^2 \big[1 - (1+\delta) z_{m,n}\big] \leq C_{3}^2 (1+\delta^{-1}) \Big(\|g\|_{L^2} + \|f\|_{H^1}\Big)^2,
\end{align*}
where
\[
z_{m,n} := \frac{m^2}{|\lambda^2+ m^2+(n+1)^2|^2}\leq \frac{1}{4a^2}\qquad(m,n\in\Z)
\]
by \eqref{eqlambdamin}. As in the previous step this yields a constant $C_{4}\geq0$ such that, for $\varepsilon>0$ such that $|a|=\frac{1}{2}+\varepsilon$,
\[
\|\lambda u\|_{H^1(\T^{2})}  \leq  C_{4}\max(1,\varepsilon^{-1})(\|f\|_{H^1(\T^{2})}+ \|g\|_{L^2(\T^{2})}).
\]
This completes the proof of \eqref{eq:toproveespilonest}.
\end{proof}

\begin{proof}[Proof of Theorem \ref{thm:WaveRen}]
The inequalities $\w_{0}(T)\geq \frac{1}{2}$ and $\w_{0}(S)\geq\frac{1}{2}$ follow from \cite{Renardy94}. Lemma \ref{lem:WaveRen} shows that the operators $-\frac{1}{2} +A$ and $-\frac{1}{2}-A$ satisfy the conditions of Theorem \ref{thm:polynomialgrowthHS} with $g(t) = \max(1/t,1)$ for $t>0$, and the latter theorem concludes the proof.
\end{proof}

\subsection*{Acknowledgements}

The authors would like to thank Yuri Tomilov for helpful comments, and the anonymous referee for carefully reading the manuscript.

\bibliographystyle{plain}
\bibliography{Bibliografiegrowth}

\end{document}